\documentclass[11pt]{article}
\usepackage[latin1]{inputenc}
\usepackage{fullpage}
\usepackage{amsfonts}
\usepackage{latexsym,amssymb,graphics,epsfig}
\usepackage{amsmath, textcomp, subfigure, color}
\usepackage{multirow}
\usepackage{natbib}

\def\FIGDIR{./}

\newcommand{\ds}{\displaystyle}
\newcommand{\RR}{\mathbb{R}}
\newcommand{\NN}{\mathbb{N}}
\newcommand{\VV}{\mathbb{V}}
\def\UU{{\mathbb U}}
\def\XX{{\mathbb X}}
\def\SCS{\VV^0}

\def\obj{{\mathbb D}} 
\def\dynamics{f}
\def\one{y}
\def\two{z}
\def\controlONE{v}
\def\controlTWO{w}
\def\dynamicsTWO{R_{\two}}
\def\dynamicsONE{R_{\one}}
\def\stockONE{B_\one} 
\def\stockTWO{B_\two} 
\def\catchONE{C_\one} 
\def\catchTWO{C_\two}
\def\tonnes{\texttt{t}}

\newcommand{\ov}{\operatornamewithlimits}
\def\mtext#1{\,\mbox{#1}\,} 
\def\defegal{:=}

\def\llower{^{\flat}}
\def\lloweropt{^{\flat,\star}}

\def\biomass{B}

\def\biology{\mathtt{Biol}}
\def\catch{h}
\def\equil{_{\textsc{e}}}
\def\SY{\mathop{\mathtt{Sust}}}
\def\MSY{\textsc{msy}}

\newtheorem{theorem}{Theorem}
\newtheorem{proposition}[theorem]{Proposition}

\newtheorem{corollary}[theorem]{Corollary}
\newtheorem{definition}[theorem]{Definition}

\newenvironment{proof}{\small{\bf Proof.}}{\hfill$\Box$\normalsize
\bigskip}

\newcounter{Appctr}
\title{Ecosystem Viable Yields}

\author{
Michel \textsc{De Lara}\footnote{%
Universit\'e Paris--Est, \textsc{CERMICS}, 
 6--8 avenue Blaise Pascal, 77455 Marne la Vall\'ee Cedex 2, France. 
Corresponding author:  delara@cermics.enpc.fr, fax +33164153586}
\and 
Eladio \textsc{Oca\~na} \footnote{\textsc{IMCA-FC}, Universidad
    Nacional de Ingenier\'{\i}a, Calle los Bi\'ologos 245, Lima
    12-Per\'u. eocana@imca.edu.pe, 
fax +511349-9838}
\and
Ricardo \textsc{Oliveros-Ramos} \thanks{\textsc{IMARPE}, 
Instituto del Mar del Per\'u,
Centro de Investigaciones en Modelado Oceanogr\'afico y Biol\'ogico Pesquero (CIMOBP), Apartado 22, Callao--Per\'u.
jtam@imarpe.gob.pe, roliveros@imarpe.gob.pe, fax +5114535053} 
\and
Jorge \textsc{Tam} \footnotemark[3]
}

\begin{document}
\maketitle

\begin{abstract}
The World Summit on Sustainable Development (Johannesburg, 2002)
encouraged the application of the ecosystem approach by 2010.
However, at the same Summit, the signatory States undertook to restore and exploit their stocks at maximum sustainable yield (MSY), a concept and practice without ecosystemic dimension, since MSY is computed species by species, on the basis of a monospecific model.
Acknowledging this gap, we propose a definition of 
``ecosystem viable yields'' (EVY)
as yields compatible i) with guaranteed biological safety levels for all time
and ii) with an ecosystem dynamics.
To the difference of MSY, this notion is not based on equilibrium, but on viability theory, which offers advantages for robustness.
For a generic class of multispecies models with harvesting, we provide explicit expressions for the EVY.
We apply our approach to the anchovy--hake couple in the Peruvian
upwelling ecosystem. 
\medskip

\end{abstract}


\section{Introduction}
\label{sec:Introduction}

Following the World Summit on Sustainable Development 
(Johannesburg, 2002), 
the signatory States undertook to restore and exploit their stocks at 
maximum sustainable yield (MSY, see \citep*{Clark:1990}).
Though being criticized for decades, MSY remains a reference.
Criticisms of MSY, like \citep*{Larkin:1977}, point out that MSY relies upon a single variable stock description (the species biomass), without age structure nor interactions with other species; 
what is more, computations are made at equilibrium.
In fisheries, one of the more elaborate method of fixing quotas, the ICES
(International Council for the Exploration of the Sea)  precautionary approach \citep*{ICES:2004}, does not assume equilibrium 
(it projects abundances one year ahead) and assumes age structure;
it remains however based on a monospecific dynamical model.
Thus, in fisheries, yields are usually defined species by species.

On the other hand, more and more emphasis is put on multispecies models
\citep*{Hollowed-Bax-Beamish-Collie-Fogarty-Livingston-Pope-Rice:2000} 
and on ecosystem management. For instance, the World Summit on Sustainable Development
encouraged the application of the ecosystem approach by 2010.
Also, sustainability is a major goal of international agreements and
guidelines to fisheries management \citep*{FAO:1999,ICES:2004}. 

Our interest is in providing conceptual insight as what could be 
\emph{sustainable yields for ecosystems}.
In this, we follow the vein of \citep*{Katz-Zabel-Harvey-Good-Levin:2003} which
introduces the concept of \emph{Ecologically Sustainable Yield}
(ESY), 
or of \citep*{Chapel-Deffuant-Martin-Mullon:2008} which 
defines yield policies in a viability approach. 
A general discussion on the ecosystem approach to fisheries may be found in
\citep*{Garcia-Zerbi-Aliaume-DoChi-Lasserre:2003}.

Our emphasis is on providing formal definition and
practical methods to design and compute such yields. 
For this purpose, our approach is not based on equilibrium calculus, 
nor on intertemporal discounted utility maximization
but on the so-called viability theory, as follows.

On the one hand, the ecosystem is described by a dynamical model
controlled by harvesting.
On the other hand, building upon \citep*{Bene-Doyen-Gabay:2001},
constraints are imposed:
catches are expected to be above given production minimal levels,
and biomasses above safety biological minimal levels.
Sustainability is here defined as the property that such constraints can be maintained for all time by appropriate harvesting strategies.

Such problems of dynamic control under constraints refer to viability 
\citep*{Aubin:1991} or invariance \citep*{Clarkeetal:1995} frameworks,
as well as to reachability of target sets or tubes for nonlinear discrete time
dynamics in~\citep*{bertsekas71minimax}.

We consider sustainable management issues formulated within
such framework  as in
\citep*{Bene-Doyen-Gabay:2001,Doyen-Bene:2003,Eisenack-Sheffran-Kropp:2006,Mullonetal:2004,Rapaport-Terreaux-Doyen:2006,DeLara-Doyen-Guilbaud-Rochet_IJMS:2007,DeLara-Doyen:2008,Chapel-Deffuant-Martin-Mullon:2008}.

A viable state is an initial condition for the
ecosystem dynamical system such that appropriate harvesting rules may drive 
the system on a sustainable path by maintaining catches and
biomasses above their respective production and biological minimal levels. 
We provide a way to characterize production minimal levels (yields) 
such that the present initial conditions are a viable state.
These yields are sustainable in the sense that they
can be indefinitely guaranteed, while making possible that the ecosystem remains in an ecologically viable zone;
we coin them \emph{ecosystem viable yields} (EVY).

Thus, the EVY can be seen as an extension of the MSY concept in two directions: 1) from equilibrium to viability (more robust);
2) from monospecies to multispecies models. 
The second claim is obvious because, as we have recalled at the beginning, 
the MSY relies upon a single variable stock description (the species biomass), without interactions with other species.
As for the first, recall that the MSY is the largest constant yield that can be taken from a single species stock over an indefinite period. By contrast, EVY are \emph{guaranteed} yields, but they are not necessarily the annual catches. Indeed, it is by an \emph{adaptive} catch policy (depending on the states of the stocks) that we shall be able to display yields  \emph{indefinitely above} the EVY. This is why we say that EVY are {guaranteed} yields, in the sense that catches cannot fall below the EVY. Viability can be seen as a robust 
extension of equilibrium: yields are not supposed to be sustained by
applying fixed stationary catches, but are minimal levels which can be
guaranteed by means of {adaptive} catch policies. 
\medskip

The paper is organized as follows.
In Section~\ref{sec:Ecosystem_sustainable_yields},
we introduce generic harvested nonlinear ecosystem models, 
and we present how preservation and production constraints are
modelled. 
Thanks to an explicit description of viable states, 
we are able to characterize sustainable yields.
These latter are not defined species by species, 
but depend on the whole ecosystem dynamics 
and on all biological minimal levels.
In Section~\ref{sec:A_numerical_application_to_the_hake-anchovy_Peruvian_ecosystem}, an illustration in ecosystem management and numerical applications are 
given for the hake--anchovy couple in the Peruvian upwelling ecosystem 
between the years 1971 and 1981.
We conclude in Section~\ref{sec:Conclusion} with possible extensions of the notion of ecosystem viable yields, on the one hand, to more general ecosystem models and, on the other hand, to bio-economic models so as to incorporate 
some economic considerations. We also discuss the limits of the EVY concept.
In the appendix,
Section~\ref{sec:Viability_and_sustainable_management} is devoted to
recalls on discrete--time viability and its 
possible use for sustainable management, while 
Section~\ref{sec:Viable_control_of_generic_nonlinear_ecosystem_models}
contains the mathematical proofs.

\section{Ecosystem viable yields}
\label{sec:Ecosystem_sustainable_yields}

After a brief recall on the notion of maximum sustainable yield (for monospecific models), 
we introduce a class of generic harvested nonlinear ecosystem models, then 
present how to define maximum sustainable yields for this class.
Next, we provide an explicit description of viable states, for which production and biological constraints can be guaranteed for all times under appropriate management.
This makes possible to define \emph{ecosystem viable yields}, compatible with biological and conservation constraints. 
We end up discussing relations between ecosystem viable yields and maximum sustainable yields.

\subsection{A brief recall on maximum sustainable yield}

We briefly sketch the principles leading to the notion of
\emph{maximum sustainable yield} (MSY) (see \citep*{Clark:1990} in continuous time and
\citep*{DeLara-Doyen:2008} in discrete time). 

Consider a single population described by its total biomass $\biomass(t)$ 
at time $t$.
Suppose that the time evolution of the biomass is given by a dynamical 
equation, a differential equation 
\( \dot{\biomass}(t)=\biology_c\big(\biomass(t)\big) \) in continuous time 
or a difference equation \( \biomass(t+1)=\biology_d\big(\biomass(t)\big) \)
in discrete time.
From this, build a Schaefer model \citep*{Schaefer:1954} by substracting
a catch term $\catch(t)$,
giving \( \dot{\biomass}(t)=\biology_c\big(\biomass(t)\big) - \catch(t) \)
or \( \biomass(t+1)= \biology_d\big(\biomass(t) - \catch(t) \big) \).
In general, to each biomass level \( \biomass\equil \) (below the carrying
capacity) corresponds a catch level \( \catch\equil=\SY(\biomass\equil) \) 
for which the biomass \( \biomass\equil \) is at equilibrium, solution of
\( \biology_c\big(\biomass\equil\big) - \catch\equil = 0 \) or
\( \biology_d\big(\biomass\equil - \catch\equil \big) = \biomass\equil \).
The \emph{maximum sustainable yield} is the largest of such
equilibrium catches: \( \MSY=\max_{\biomass\equil}\SY(\biomass\equil) \).






\subsection{Ecosystem biomass dynamical model}

For simplicity, we consider a model with two species, but it can 
be easiliy extended to $N$ species in interaction.
Each species is described by its biomass:
the two--dimensional state vector $(\one,\two)$ represents the biomasses
of both species. 
The two--dimensional control $(\controlONE,\controlTWO)$ 
comprises the harvesting effort for each species, 
respectively. The catches are thus $\controlONE \one$ and 
$\controlTWO \two$ (measured in
biomass).\footnote{%
In fact, any expression of the form $c(\one,\controlONE)$,
instead of $\controlONE \one$, would fit for
the catches in the following Proposition~\ref{pr:predator_prey} as
soon as $\controlONE \mapsto c(\one,\controlONE)$ 
is strictly increasing and goes
from $0$ to $+\infty$ when $\controlONE$ goes from $0$ to $+\infty$.
The same holds for $d(\two,\controlTWO)$ instead of $\controlTWO \two$.
}
The discrete--time control system  we consider is 
\begin{equation}
\left\{ \begin{array}{rcl}
 \one(t+1) &=& \one(t) \dynamicsONE\big(\one(t),\two(t),\controlONE(t)\big) \; , \\
 \two(t+1) &=& \two(t) \dynamicsTWO\big(\one(t),\two(t),\controlTWO(t)\big) \; ,
\end{array} \right.
\label{eq:ecosystem}
\end{equation}
where $t$ stand for time (typically, periods are years),
and where $\dynamicsONE:\RR^3 \to \RR$ and $\dynamicsTWO:\RR^3 \to \RR$ are two
functions representing growth factors
(the growth rates being $\dynamicsONE -1$ and $\dynamicsTWO -1 $). 
This model is generic in that no explicit or analytic assumptions are made on how the growth factors $\dynamicsONE$ and $\dynamicsTWO$ indeed depend
upon both biomasses $(\one,\two)$.

In the above model, each species is harvested by a
specific device: one species, one harvesting effort. This covers the 
multioutput settings case (e.g. several species in trophic interactions and targeted by the same fishing gear). Indeed, for this it suffices to state that both efforts are identical: 
$\controlONE(t) = \controlTWO(t) $ for all \( t=t_{0},t_{0}+1,\ldots \).


\subsection{Preservation and production sustainability}

We now propose to define sustainability as the ability to respect
preservation and production minimal levels for all times, building upon the original approach of \citep*{Bene-Doyen-Gabay:2001}.
Let us be given
\begin{itemize}
 \item on the one hand, 
\emph{minimal biomass levels} 
$\stockONE\llower\geq 0$, $\stockTWO\llower\geq 0$, 
one for each species, 
 \item on the other hand, \emph{minimal catch levels} 
$\catchONE\llower\geq 0$, $\catchTWO\llower\geq 0$,
one for each species.
\end{itemize}

A couple $(\one_0,\two_0)$ of initial biomasses is said to be a \emph{viable state}
if there exist appropriate harvesting efforts (controls)
$\big(\controlONE(t),\controlTWO(t)\big), \;
t=t_{0},t_{0}+1,\ldots$ such that the state path 
$\big(\one(t),\two(t)\big), \; t=t_{0},t_{0}+1,\ldots$ 
starting from $\big(\one(t_0),\two(t_0)\big)=(\one_0,\two_0)$ 
satisfies the following goals:
\begin{itemize}
\item {preservation} (minimal biomass levels)
\begin{equation}
\mtext{biomasses: } 
{ \one(t) \geq \stockONE\llower \; , \quad \two(t) \geq \stockTWO\llower }
\; , \quad \forall t=t_{0},t_{0}+1,\ldots
\label{eq:guaranteed_biomass}
\end{equation}
 \item and {production} requirements 
(minimal catch levels)
\begin{equation}
\mtext{catches: } 
{
\controlONE(t) \one(t) \geq \catchONE\llower \; , \quad
\controlTWO(t) \two(t) \geq \catchTWO\llower 
\; , \quad \forall t=t_{0},t_{0}+1,\ldots 
}
\label{eq:guaranteed_catch}
\end{equation}
\end{itemize}
The set of all viable states is called the \emph{viability kernel} 
\citep*{Aubin:1991}. 
Characterizing viable states 
makes it possible to test whether or not minimal biomasses and catches
can be guaranteed for all time. 

Here, by \emph{guaranteed}, we mean that the yields have to \emph{indefinitely above} the EVY, as reflected in the inequalities~\eqref{eq:guaranteed_biomass} and  \eqref{eq:guaranteed_catch}. We insist on the fact that, in the above definition of viable states, we say ``there exist appropriate harvesting efforts (controls)'': that is, to guarantee the EVY, we need to resort to  \emph{adaptive} catch policies (depending on the states of the stocks
as in Corollary~\ref {cor:viable-contr} in the Appendix).
Viability can be seen as a robust extension of equilibrium: 
yields are not supposed to be sustained by
applying fixed stationary catches, but are minimal levels which can be
guaranteed by means of {adaptive} catch policies. 

Notice that, in the multioutput settings case, we need to add the constraint 
$\controlONE(t) = \controlTWO(t) $ for all \( t=t_{0},t_{0}+1,\ldots \).
Therefore, with an additional constraint, the set of viable states in the multioutput settings case will be smaller than the one considered above.

The following definition summarizes useful and natural properties 
required for the growth factors in the ecosystem model.
\begin{definition}
We say that growth factors $\dynamicsONE$ and
$\dynamicsTWO$ in the ecosystem model~\eqref{eq:ecosystem} 
are \emph{nice} if 
the function $\dynamicsONE:\RR^3 \to \RR$ is 
continuously decreasing\footnote{%
In all that follows, a mapping $\varphi: \RR \to \RR$ is said to be
increasing if $x \geq x' \Rightarrow \varphi(x)  \geq \varphi(x')$. 
The reverse holds for decreasing. Thus, with
this definition, a constant mapping is both increasing and decreasing.
} in the harvesting effort $\controlONE$  
and satisfies $\lim_{\controlONE \to +\infty}
\dynamicsONE(\one,\two,\controlONE)\leq 0$,
and if $\dynamicsTWO:\RR^3 \to \RR$ is 
continuously decreasing in the harvesting effort $\controlTWO$,
and satisfies $\lim_{\controlTWO \to +\infty}
\dynamicsTWO(\one,\two,\controlTWO)\leq 0$. 
\end{definition}

The following Proposition~\ref{pr:predator_prey} gives an explicit
description of the viable states, 
under some conditions on the minimal levels. 
Its proof is given in \S~\ref{ss:Expression_of_the_viability_kernel} in the Appendix.

The Proposition~\ref{pr:predator_prey} may easily be extended to $N$ species in interaction as long as  each species is harvested by a
specific device: one species, one harvesting effort.
However, it is not valid in the multioutput settings case. 
Indeed, it is crucial to have two distinct controls 
$\controlONE(t)$ and $\controlTWO(t) $ for the proof.
Assuming that 
$\controlONE(t) = \controlTWO(t) $ for all \( t=t_{0},t_{0}+1,\ldots \)
would require other types of calculations for the viability kernel. This is out of the scope of this paper.

\begin{proposition}
Assume that the growth factors in the ecosystem model~\eqref{eq:ecosystem} are nice.
If the biomass minimal levels $\stockONE\llower $, $\stockTWO\llower $,
and the catch minimal levels $\catchONE\llower $, $\catchTWO\llower $
are such that the following growth factors values are greater than one
\begin{equation}
\dynamicsONE(\stockONE\llower,\stockTWO\llower,
\frac{\catchONE\llower}{\stockONE\llower})\geq 1
\mtext{ and } 
\dynamicsTWO(\stockONE\llower,\stockTWO\llower,
\frac{\catchTWO\llower}{\stockTWO\llower})\geq 1 \; ,
\label{eq:favorable_conditions}
\end{equation}
then the viable states are all the couples $ (\one,\two)$  of biomasses 
such that
\begin{equation}
\one\geq \stockONE\llower, \quad   \two\geq \stockTWO\llower, \quad   
\one\dynamicsONE(\one,\two,\frac{\catchONE\llower}{\one})\geq \stockONE\llower, \quad 
\two\dynamicsTWO(\one,\two,\frac{\catchTWO\llower}{\two})\geq \stockTWO\llower 
\; . 
\end{equation}
\label{pr:predator_prey}
\end{proposition}

Let us comment the assumptions of Proposition~\ref{pr:predator_prey}.
That the growth factors are decreasing with respect to the
harvesting effort is a natural assumption.
Conditions~\eqref{eq:favorable_conditions} mean that, at the point 
$(\stockONE\llower,\stockTWO\llower)$ and applying efforts 
$u\llower=\frac{\catchONE\llower}{\stockONE\llower}$,
$v\llower=\frac{\catchTWO\llower}{\stockTWO\llower}$,
the growth factors are greater than one, hence both populations grow;
hence, it could be thought that computing 
viable states 
is useless since everything looks fine. 
However, if all is fine at the point $(\stockONE\llower,\stockTWO\llower)$, 
it is not obvious that this also goes for a larger domain.
Indeed, the ecosystem dynamics given by~\eqref{eq:ecosystem}
has no  monotonocity properties that would allow to extend a result 
valid for a point to a whole domain.
What is more, if continuous-time viability results mostly relies upon
assumptions at the frontier of the constraints set, this is no longer
true for discrete-time viability.


\subsection{Ecosystem viable yields}

Considering that minimal biomass conservation levels 
are given first (for prominent biological issues), 
we shall now examine conditions for the existence of minimal catch levels 

First, we define (when they exist) the \emph{ecosystem viable yields}.

\begin{definition}
Let biomass conservation minimal levels 
$\stockONE\llower\geq 0$, $\stockTWO\llower\geq 0$ be given.
Suppose that the growth factors in the ecosystem model~\eqref{eq:ecosystem} are nice, and that they take values greater than one in the absence of harvesting, namely:
\begin{equation}
\dynamicsONE(\stockONE\llower,\stockTWO\llower,0)\geq 1 
\mtext{ and }
\dynamicsTWO(\stockONE\llower,\stockTWO\llower,0)\geq 1 \; .
\label{eq:favorable_conditions_0}
\end{equation}
Define equilibrium catches as the largest nonnegative\footnote{%
Such catches are nonnegative because the growth factors in the ecosystem model~\eqref{eq:ecosystem} are nice, hence continuously decreasing in the harvesting effort, and by~\eqref{eq:favorable_conditions_0}.}
catches 
\( \catchONE\lloweropt , \catchTWO\lloweropt \) such that 
\begin{equation}
\dynamicsONE(\stockONE\llower,\stockTWO\llower,
\frac{\catchONE\lloweropt}{\stockONE\llower} ) = 1 
\mtext{ and }
\dynamicsTWO(\stockONE\llower,\stockTWO\llower,
\frac{\catchTWO\lloweropt}{\stockTWO\llower} ) = 1 \; .
\end{equation}
For a couple $(\one_0,\two_0)$ of biomasses, define (when they exist) the \emph{ecosystem viable yields} (EVY)
\( \catchONE\lloweropt(\one_0,\two_0) \) 
and \( \catchTWO\lloweropt(\one_0,\two_0) \) by
\begin{equation}
\left\{ \begin{array}{rcl}
\catchONE\lloweropt(\one_0,\two_0) &\defegal & 
\max \{ \catchONE \in [ 0 , \catchONE\lloweropt ] \mid 
\one_0\dynamicsONE(\one_0,\two_0,\frac{\catchONE}{\one_0})\geq
\stockONE\llower \} \; , \\[3mm]
\catchTWO\lloweropt(\one_0,\two_0) &\defegal & 
\max \{ \catchTWO  \in [ 0 , \catchTWO\lloweropt ] \mid
\two_0\dynamicsTWO(\one_0,\two_0,\frac{\catchTWO}{\two_0})\geq 
\stockTWO\llower \} \; .
\end{array} \right.
\label{eq:state_maximal_catches}
\end{equation}
\end{definition}

The term \emph{ecosystem viable yields} is justified by the following 
Proposition~\ref{pr:favorable_conditions_equilibrium}. 

\begin{proposition}
Assume that the growth factors in the ecosystem model~\eqref{eq:ecosystem} are nice.
For a couple $(\one_0,\two_0)$ of biomasses above preservation minimal levels -- that is, $ \one_0 \geq \stockONE\llower$ and 
$\two_0 \geq \stockTWO\llower$ -- and satisfying 
\begin{equation}
 \one_0\dynamicsONE(\one_0,\two_0,0)\geq \stockONE\llower 
\mtext{ and }
\two_0\dynamicsTWO(\one_0,\two_0,0)\geq \stockTWO\llower \; ,
\label{eq:conditions_initial_point}
\end{equation}
the {ecosystem viable yields} $\catchONE\lloweropt(\one_0,\two_0) $ and 
$\catchTWO\lloweropt(\one_0,\two_0)$
in~\eqref{eq:state_maximal_catches} are well defined.

What is more, consider catches
$\catchONE\llower$ and $\catchTWO\llower$ lower than these 
{ecosystem viable yields}, that is, 
$0 \leq \catchONE\llower \leq \catchONE\lloweropt(\one_0,\two_0) $
and $0 \leq \catchTWO\llower \leq \catchTWO\lloweropt(\one_0,\two_0)$.
Then, starting from the initial biomasses 
\( (\one(t_0),\two(t_0)) = (\one_0,\two_0) \),
there exists appropriate harvesting paths
which provide, for all time, at least the \emph{sustainable yields}
$\catchONE\llower$ and $\catchTWO\llower$ and which guarantee 
that biomass conservation minimal levels 
$\stockONE\llower\geq 0$, $\stockTWO\llower\geq 0$ are respected 
for all time.
\label{pr:favorable_conditions_equilibrium}
\end{proposition}

From the practical point of view, the upper quantities 
\( \catchONE\lloweropt(\one_0,\two_0) \) and
\( \catchTWO\lloweropt(\one_0,\two_0) \) in
\eqref{eq:state_maximal_catches} cannot be seen as 
catches targets, but rather as \emph{crisis limits}.
Indeed, the closer to them, the more vulnerable,
since the initial point is close to the viability kernel boundary.

Notice that the yield $\catchONE\lloweropt(\one_0,\two_0)$ depends,
first,
on both species biomasses $(\one_0,\two_0)$,
second,
on both conservation minimal levels $\stockONE\llower$ 
and $\stockTWO\llower$,
third,
on the ecosystem model by the growth factor $\dynamicsONE$;
the same holds for $\catchTWO\lloweropt(\one_0,\two_0)$. 
Thus, these yields are designed jointly on the basis of the whole
ecosystem model and of all the conservation minimal levels; this is why we coined them \emph{ecosystem} viable yields.

This observation may have practical consequences. 
Indeed, the catches guaranteed for one species depend not only on the 
biological minimal level of the same species, but on the other species.
For instance, in the Peruvian upwelling ecosystem, it is customary to
increase the biological minimal level of the anchovy before an 
El Ni\~{n}o event, but without explicitely considering to lower catches
of other species. Our analysis stresses the point that minimal levels have
to be designed globally to guarantee sustainability for the whole
ecosystem.

\subsection{Ecosystem viable yields and maximum sustainable yields}

Now, we show how \emph{ecosystem viable yields} are related to 
\emph{maximum sustainable yields}.

An \emph{equilibrium} of the ecosystem model~\eqref{eq:ecosystem} is 
a couple $(\one\equil,\two\equil)$ of biomasses (state) and a couple 
$(\controlONE\equil,\controlTWO\equil)$ of harvesting efforts (control)
satisfying 
\begin{equation}
\left\{ \begin{array}{rcl}
 \one\equil &=& \one\equil \dynamicsONE\big(\one\equil,\two\equil,\controlONE\equil\big)  \; , \\
 \two\equil &=& \two\equil \dynamicsTWO\big(\one\equil,\two\equil,\controlTWO\equil\big) \; .
\end{array} \right.
\label{eq:equil_ecosystem}
\end{equation}
The \emph{maximum sustainable yields}, $\MSY_\one$ for species $\one$ 
and  $\MSY_\two$ for species $\two$, are given by 
\begin{equation}
\MSY_\one \defegal \max_{\controlONE\equil,\controlTWO\equil} 
\controlONE\equil \one\equil \mtext{ and }
\MSY_\two \defegal \max_{\controlONE\equil,\controlTWO\equil} 
\controlTWO\equil \two\equil \; . 
\label{eq:MSY_ecosystem}
\end{equation}
They must be jointly defined because the ecosystem equilibrium equations~\eqref{eq:equil_ecosystem} couple all variables.  

Say that the maximum sustainable yields $\MSY_\one$ and $\MSY_\two$ are 
\emph{viable maximum sustainable yields} if the corresponding 
biomasses equilibrium values \( \one\equil \) and \(\two\equil \) 
are such that \( \one\equil \geq \stockONE\llower \) and 
\(\two\equil \geq \stockTWO\llower \).
In this case, $\MSY_\one$ and $\MSY_\two$ are ecosystem viable yields
for the couple $(\one\equil,\two\equil)$ of initial biomasses:
indeed, the stationary harvest strategy 
\( \controlONE(t)=\controlONE\equil \) and
\( \controlTWO(t)=\controlTWO\equil \) drives 
the ecosystem model~\eqref{eq:ecosystem} at equilibrium 
$(\one\equil,\two\equil)$ which satisfies the conservation minimal levels 
\( \one\equil \geq \stockONE\llower \) and 
\(\two\equil \geq \stockTWO\llower \).

Notice that the maximum sustainable yields $\MSY_\one$ and $\MSY_\two$ are defined independently of the initial biomasses, whereas the {ecosystem viable yields} (EVY)
\( \catchONE\lloweropt(\one_0,\two_0) \) 
and \( \catchTWO\lloweropt(\one_0,\two_0) \) explicitely depend upon them.

\section{Numerical application to the hake--anchovy couple in the 
Peruvian upwelling ecosystem (1971--1981)}
\label{sec:A_numerical_application_to_the_hake-anchovy_Peruvian_ecosystem}

We provide a viability analysis of the
hake--anchovy Peruvian fisheries between the years 1971 and 1981.
For this, we shall consider a discrete-time Lotka--Volterra model 
for the couple anchovy (prey $\one$) and hake (predator $\two$), 
then provide an explicit description of 
viable states. 

We warn the reader that our emphasis is not on developing a ``knowledge'' biological model to ``faithfully'' describe the complexity of the Peruvian upwelling ecosystem.
This formidable task is out of our competencies, and is not 
necessary for our analysis.
Indeed, our approach makes use, not of ``knowledge'' models,
but of ``action'' models; these are small, compact, models which capture essentials features of the system in what concern decision-making. 
In our case, we needed a compact model able to put in consistency 
biomass and catches yearly data between the years 1971 and 1981.
We chose a discrete--time Lotka--Volterra model, despite 
well-known criticisms as candidate for a ``knowledge'' biological model \citep*{Hall:1988, Murray:2002}, but for its compactness 
qualities and for the reasonable fit (see Figure~\ref{fig:anchovy_hake_non_linear.eps}).

\subsection{Viable states and 
ecosystem sustainable yields for a Lotka--Volterra system}

Consider the following discrete--time Lotka--Volterra system of
equations with density--dependence in the prey
\begin{equation}
\left\{ \begin{array}{rcl}
\one(t+1) &=& \displaystyle 
R\one(t)-\frac{R}{\kappa}\one^2(t)-\alpha \one(t)\two(t) 
- \controlONE(t) \one(t) \; ,\\[3mm]
\two(t+1) &=& L\two(t)+\beta \one(t)\two(t)-\controlTWO(t)\two(t) \; ,
\end{array} \right.
\label{eq:LV_with_dd}
\end{equation}
where $R>1$, $0<L<1$, $\alpha >0$, $\beta >0$ and
$\kappa=\tfrac{R}{R-1}K$, with $K>0$ the carrying capacity for prey. 
In the dynamics~\eqref{eq:ecosystem}, we identify 
$\dynamicsONE(\one,\two,\controlONE)=
R-\frac{R}{\kappa}\one-\alpha \two - \controlONE $ and
$\dynamicsTWO(\one,\controlTWO)=L+\beta \one- \controlTWO $.

By Proposition~\ref{pr:favorable_conditions_equilibrium}, 
we obtain that, for any initial point $(\one_0,\two_0)$ such that 
\begin{equation}
\one_0 \geq \stockONE\llower \; , \quad
\two_0 \geq \stockTWO\llower \; , \quad
\one_0 ( R-\frac{R}{\kappa}\one_0-\alpha \two_0 ) \geq \stockONE\llower
\; ,
\label{eq:LV_viability}
\end{equation}
the ecosystem sustainable yields are given by
\begin{equation}
\left\{ \begin{array}{rcl}
\catchONE\lloweropt(\one_0,\two_0) &=& 
\min \left\{ \stockONE\llower(R-\tfrac{R}{\kappa}\stockONE\llower -\alpha \stockTWO\llower) - \stockONE\llower , 
\one_0 ( R - \frac{R}{\kappa}\one_0 - \alpha\two_0) -
\stockONE\llower  \right\}   \\[3mm]
\catchTWO\lloweropt(\one_0,\two_0) &=& 
\stockTWO\llower(L + \beta \stockONE\llower - 1) \; .
\end{array} \right.
\label{eq:LV_ecosystem_sustainable_yields}
\end{equation}
In other words, if viably managed, the ecosystem could produce at least 
$ \catchONE\lloweropt(\one_0,\two_0) $ 
and $\catchTWO\lloweropt(\one_0,\two_0) $, while respecting
biological minimal levels $\stockONE\llower$ and $\stockTWO\llower$.

\subsection{A viability analysis of the
hake--anchovy Peruvian fisheries between the years 1971 and 1981}

The Peruvian upwelling ecosystem is extremely productive and dominated by anchovy (\emph{Engraulis ringens}) dynamics. It is well known that anchovy fisheries are very sensitive to environmental variability~\citep*{Checkley:2009, Sun_etal:2001}, and the Peruvian anchovy is subject to environmental perturbations such as El Ni\~{n}o Southern Oscillation (ENSO) variability~\citep*{Bertrand_etal:2004}). However, for this simple predator-prey model using hake and anchovy, we have assumed that no uncertainties affect the ecosystem dynamics. 
Indeed, we feel that we have to go step by step in introducing the EVY concept, first focusing on the deterministic case. 
Thus, the period between the years 1971 and 1981 is suitable for this first version of the model, due to the absence of strong El Ni\~{n}o events in the middle of the period. Furthermore, the long-term dynamics of the Peruvian upwelling ecosystem is dominated by shifts between alternating anchovy and sardine regimes that restructure the entire ecosystem~\citep*{Alheit-Niquen:2004}. The period from 1970 to 1985 was characterized by positive temperature anomalies and low anchovy abundances, after the anchovy collapse in 1971~\citep*{Alheit-Niquen:2004}, so the competition between the fishery and hake was reduced due to low anchovy catches and anchovy mortality due to hake predation increase. Particularly, the changes in the ecosystem after 1971 led to an increase of five times, in average, in the predation rates of hake over anchovy between 1971 to 1980, with a peak in 1977~\citep*{Pauly-Palomares:1989}. 

Between the years 1971 and 1981, we have 11 couples of biomasses, and the same for catches. The 5~parameters of the Lotka-Volterra model are estimated minimizing a weighted residual squares sum function using a conjugate gradient method, with central derivatives. Estimated parameters and comparisons of observed and simulated biomasses are shown in Figure~\ref{fig:anchovy_hake_non_linear.eps}.

\begin{figure}
\centering
\begin{tabular}{cc}
\subfigure[Anchovy]{ \includegraphics[width=0.45\textwidth]%
{\FIGDIR 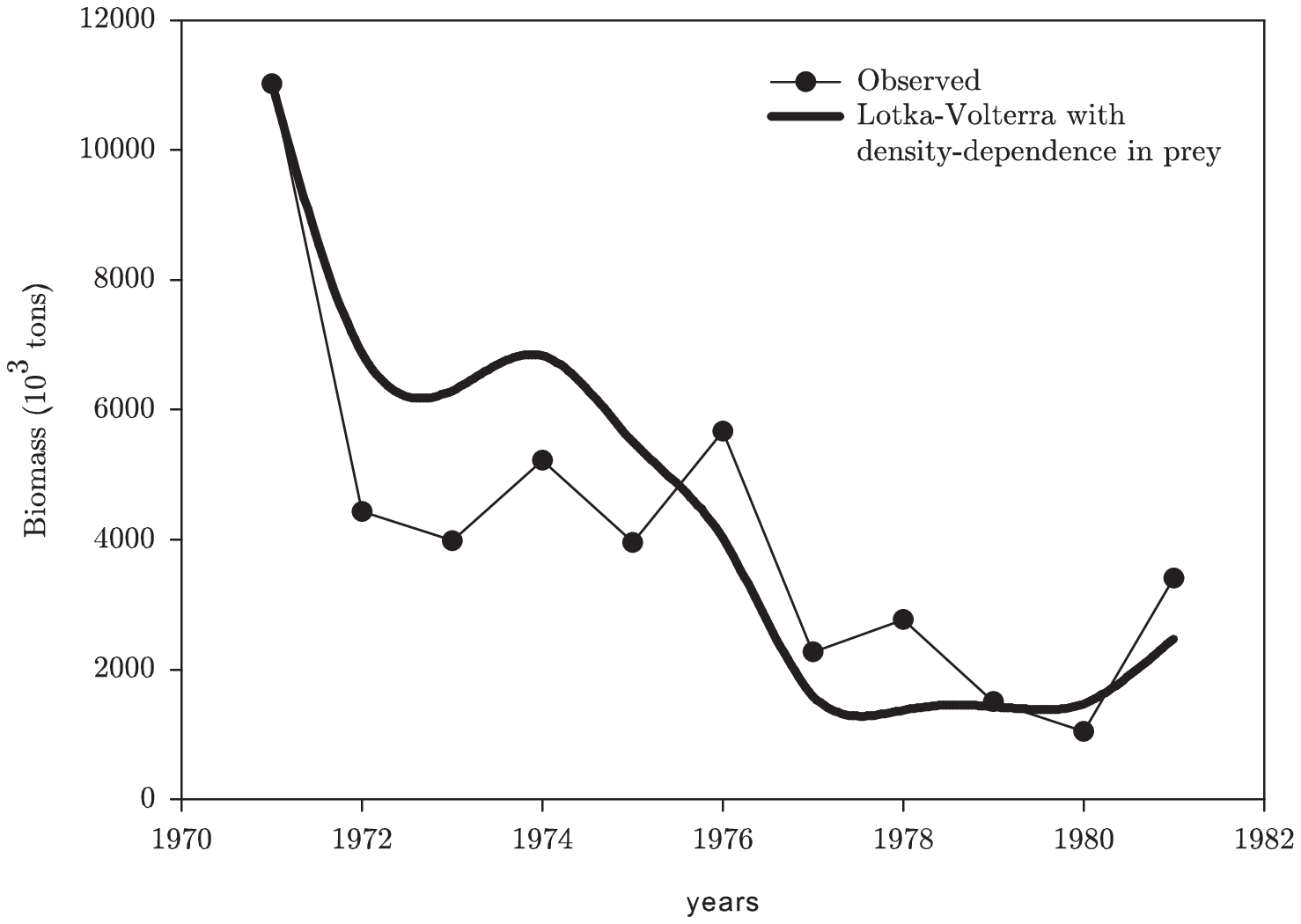}} &
\subfigure[Hake]{\includegraphics[width=0.45\textwidth]%
{\FIGDIR 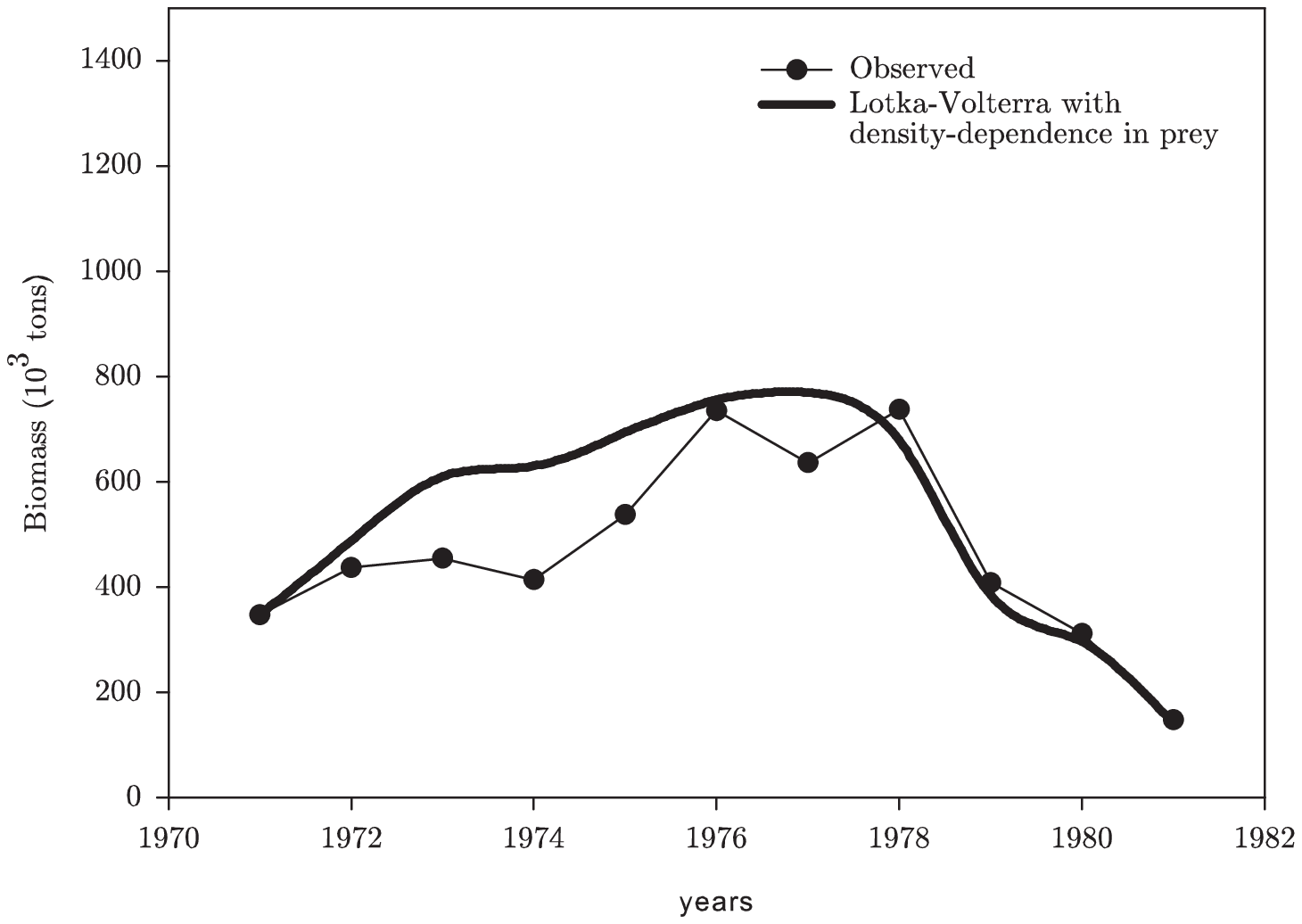}}
\end{tabular}
\caption{Comparison of observed and simulated biomasses of anchovy and
  hake using a Lotka--Volterra model with density-dependence in the
  prey (1971--1981). Model parameters are 
$R=2.25$~year$^{-1}$, 
$L=0.945$~year$^{-1}$, 
$\kappa=67,113~\times 10^3$~\tonnes\ ($K=37,285~\times 10^3$~\tonnes),
$\alpha=1.220\times 10^{-6}$~\tonnes$^{-1}$,
$\beta= 4.845\times 10^{-8}$~\tonnes$^{-1}$. }
\label{fig:anchovy_hake_non_linear.eps}
\end{figure}

We consider values of $\stockONE\llower=7,000,000$~\tonnes\ and
$\stockTWO\llower=200,000$~\tonnes\ for minimal biomass levels 
\citep*{IMARPE:2000, IMARPE:2004}. 
Conditions~\eqref{eq:LV_viability} are satisfied 
and the expressions~\eqref{eq:LV_ecosystem_sustainable_yields} give 
the ecosystem viable yields (EVY) 
\begin{equation}
\catchONE\lloweropt(\one_0,\two_0) =5,399,000~\tonnes\  \mtext{ and }
\catchTWO\lloweropt(\one_0,\two_0)=56,800~\tonnes\ \; .
\end{equation}
In other words, such yields were theoreticaly susceptible to be
guaranteed in a   sustainable way starting from year 1971.
In reality, the catches of year 1971 were very high and 
the biomasses trajectories were well below the biological minimal levels for
fourteen years. 

The $4,250,000$~\tonnes\ anchovy quota and the 
$55,000$~\tonnes\ hake quota, respectively, established for the year 2006
\citep*{PRODUCE:2005},
or the $5,000,000$~\tonnes\ anchovy quota 
and the $35,000$~\tonnes\ hake quota,
respectively, established for the year 2007 \citep*{PRODUCE:2006a}
 are rather close to the EVY
\( \catchONE\lloweropt(\one_0,\two_0) =5,399,000\)~\tonnes\  and
\( \catchTWO\lloweropt(\one_0,\two_0)=56,800 \)~\tonnes.
Thus, our approach provides reasonable figures.\footnote{%
At this stage, we do not claim that the figures may be proposed as 
yields for the present management of hake--anchovy Peruvian fisheries.
Indeed, our computations of EVY rely upon a dynamical model adjusted for 
some thirty years ago. To propose EVY, we should first dispose of a dynamical model adapted to the current situation, because it ought to reflect the new ecosystem functioning and the depleted state of stocks 
\citep*{Ballon-Wosnitza-Mendo-Guevara-Carrasco-Bertrand:2008}.
This is beyond the scope of this paper.}


\section{Conclusion}
\label{sec:Conclusion}

We have defined the notion of sustainable yields for ecosystem,
and provided ways to compute them by means of a viability analysis of
generic ecosystem models with harvesting. 
Our analysis stresses the point that yields should certainly be designed globally, and not species by species as in the current practice, to guarantee sustainability for the whole ecosystem. 
Our results have then been applied to a Lotka--Volterra model using the 
anchovy--hake couple in the Peruvian upwelling ecosystem.
Despite simplicity\footnote{%
In addition to hake, there are other important
predators of anchovy in the Peruvian upwelling ecosystem, such as
mackerel and horse mackerel, seabirds and pinnipeds, which were not
considered. Also,
anchovy has been an important prey of hake, but other prey species have
been found in the opportunistic diet of hake \citep*{Tam:2006}
}
of the models considered, our approach has provided reasonable figures 
and new insights: it may be a mean of designing
sustainable yields from an ecosystem point of view. 

We now discuss the limits of the EVY concept, as presented in this paper: 
application to biomass ecosystem models without age or spatial structure, no economic consideration, no uncertainties. We stress that EVY is a flexible concept, and we hint at possible extensions to incorporate the missing dimensions listed above.

The framework we propose is not restricted to two populations, each described by its global biomass, but it may be adapted to several species, each described by a vector of abundances at age, or by vectors of abundances at age
for each patch in a spatial model, etc.
Suppose that the time evolution is given by a
dynamical equation reflecting ecosystemic interactions and driven by efforts or by catches. 
Suppose that minimal safety levels (reference points) are fixed for biological indicators like spawning stock biomass, 
abundances at specific ages, etc. (such reference points for biological indicators like spawning stock biomass are generally given by international  bodies like the ICES, or nationally).  
Ecosystem viable yields are minimal harvests for each species which  
can be guaranteed for all times while respecting the above minimal safety levels for biological indicators for all times too. 

It is often objected with reason that the MSY concept is developed without any economic consideration. As presented here, the EVY suffers the same criticism. 
However, the EVY concept is flexible enough to incorporate 
some economic considerations. For instance, upper bounds for fishing costs may be incorporated as constraints to be satisfied for all time, aside with minimal biomass levels. In this sense, EVY will be guaranteed yields 
compatible with biological and economic restrictions.

As presented here, the EVY framework supposes that no uncertainties affect the ecosystem dynamics. Though we have the tools to tackle such an important issue
(see stochastic viability in~\citep*{DeLara-Doyen:2008,DeLara-Martinet:2009,Doyen-DeLara:2010}), 
we feel that we have to go step by step. 
This paper introduces the EVY concept in the deterministic case,
providing an extension of the MSY concept in two directions: 
from equilibrium to viability (more robust), from monospecies to multispecies models. The extension to the uncertain case is currently under investigation.

Thus, control and viability theory methods have allowed us to 
introduce ecosystem considerations, such as multispecies and
multiobjectives, and have contributed to integrate the long term dynamics, 
which is generally not considered in conventional fishery management.  

\bigskip

\paragraph*{Acknowledgments.}

This paper was prepared within the MIFIMA (Mathematics, 
  Informatics and Fisheries Management) international research network;
we thank CNRS, INRIA and the French Ministry of Foreign Affairs 
for their funding and support through the regional cooperation program
STIC--AmSud.
We thank the staff of the Peruvian Marine Research Institute (IMARPE),
especially Erich Diaz and Nathaly Vargas for discussions on anchovy and
hake fisheries. 
We thank Sophie Bertrand and Arnaud Bertrand from IRD at IMARPE for 
their insightful comments. 
We also thank Yboon Garcia (IMCA-Peru and CMM-Chile)
for a discussion on the ecosystem model case.
We are particularly indebted to the reviewer who, by his/her comments and questions, helped us improve the presentation.

\appendix

\section{Discrete--time viability}
\label{sec:Viability_and_sustainable_management}

Let us  consider a nonlinear control system  described in discrete--time
by the dynamic equation
\begin{equation}
\left\{ \begin{array}{l}
x(t+1)=\dynamics\big(x(t),u(t)\big) \mtext{ for all } t \in \NN ,\\
x(0)=x_0 \mtext{ given,}
\end{array} \right .
\label{eq:generaldyn}
\end{equation}
where  the \emph{state variable} $x(t)$ belongs to  the finite
dimensional state space $\XX=\RR^{n_{\XX}}$, the \emph{control variable
} $u(t)$ is an element of the \emph{control set} $\UU=\RR^{n_{\UU}}$
while the \emph{dynamics} $\dynamics$ maps $\XX \times \UU$ into  $\XX$.

A controller or a decision maker  describes ``acceptable
configurations of the system'' through a set 
$\obj \subset \XX \times \UU$ termed the \emph{acceptable set}
\begin{equation}\label{eq:constraint}
\big(x(t),u(t)\big) \in \obj \mtext{ for all } t \in \NN \; ,
\end{equation}
where $\obj $ includes both system states and controls
constraints. 

The \emph{state constraints set} $\SCS$ associated with
$\obj $ is obtained by
projecting the acceptable set $\obj $ onto the state space $\XX$:
\begin{equation}
\SCS \defegal {\rm Proj}_{\XX}(\obj ) =
\{x \in \XX \mid \exists u \in \UU \, , \, (x,u) \in \obj \} \; .
\label{eq:state_constraints_set}
\end{equation}

Viability is defined as the ability to choose,
at each time step $t \in \NN$, a control $u(t) \in \UU$ such that the system
configuration remains acceptable.
More precisely, viability occurs when the following set of 
initial states is not empty:
\begin{equation}
\VV(\dynamics,\obj ) \defegal \left\{x_0\in \XX \left|
\begin{array}{l}\exists\; (u(0),u(1), \ldots ) \mtext{ and } (x(0),x(1),
\ldots )\\
\mtext{ satisfying } \eqref{eq:generaldyn} \mtext{ and }\eqref
{eq:constraint}
\end{array} \right.
\right\} \; .
\label{eq:viability_kernel}
\end{equation}
The set $\VV(\dynamics,\obj )$ is called the \emph{viability kernel} 
\citep*{Aubin:1991} associated
with the dynamics $\dynamics$ and the acceptable set $\obj $.
By definition, we have $\VV(\dynamics,\obj) \subset \SCS
={\rm Proj}_{\XX}(\obj )$ but, in general, the
inclusion is strict. For a decision maker or control designer, knowing
the viability kernel has practical interest since it describes the
initial states for which controls can be found that maintain the
system in an acceptable configuration forever.
However, computing this kernel is not an easy task in general.
\bigskip

We now focus on some tools to achieve viability.
A subset $\VV$ is said to be
\emph{weakly invariant}  for the dynamics $\dynamics$ in the acceptable set
$\obj $, or a \emph{viability domain} of $\dynamics$ in $\obj $, if
\begin{equation}
\forall x \in \VV  \, , \quad \exists u \in \UU  \, , \quad
(x,u) \in \obj  \mtext{ and } \dynamics(x,u) \in \VV \; .
\label{eq:viability_domain}
\end{equation}
That is, if one starts from $\VV$, an acceptable control may
transfer the state in $\VV$.
Moreover, according to viability theory \citep*{Aubin:1991}, the viability
kernel $\VV(\dynamics,\obj )$ turns out to be the union of all viability
domains, or also the largest viability domain:
\begin{equation}
\VV(\dynamics,\obj)= \bigcup \biggl\{ \VV,\;\VV \subset \SCS,\;\VV\mtext{ 
viability domain for $\dynamics$ in ${\obj}$}  \biggr\}  \; .
\label{eq:viability_kernel_union}
\end{equation}
\emph{Viable controls} are those controls $u \in \UU $ such that 
\( (x,u) \in \obj  \) and \( \dynamics(x,u) \in \VV(\dynamics,\obj) \). 

A major interest of such a property lies in the fact that
any viability domain for the dynamics $\dynamics$ in the acceptable set
$\obj $ provides a \emph{lower approximation} of the viability kernel.
An \emph{upper approximation} $\VV_k$ of the viability kernel
is given by the so called
\emph{viability kernel until time $k$
associated with $\dynamics$ in $\obj $}:
\begin{equation}
\VV_{k} \defegal \left\{x_0\in \XX \left|
\begin{array}{l}\exists\; (u(0),u(1), \ldots, u(k) ) \mtext{ and }
(x(0),x(1), \ldots, x(k) )\\
\mtext{ satisfying } \eqref{eq:generaldyn} \mtext{ for } t=0,\ldots,k-1 \\
\mtext{ and }\eqref{eq:constraint} \mtext{ for } t=0,\ldots,k
\end{array} \right.
\right\} \; .
\label{eq:Vtau}
\end{equation}
We have
\begin{equation}
\VV(\dynamics,\obj ) \subset \VV_{k+1}
\subset \VV_{k} \subset \VV_0 = \SCS \, \mtext{ for all } k \in \NN \; .
\label{eq:inclusionsVtau}
\end{equation}
It may be seen by induction that the decreasing sequence of
viability kernels until time $k$ satisfies
\begin{equation}
\VV_0 = \SCS \mtext{ and }
\VV_{k+1} = \left\{ x \in \VV_k \; \left| \; \exists u \in \UU \, ,
\, (x,u) \in \obj \mtext{ and } \dynamics(x,u) \in \VV_k
\right. \right\} \; .
\label{eq:induction}
\end{equation}
By~\eqref{eq:inclusionsVtau},
such an algorithm provides approximation from above of the
viability kernel as follows:
\begin{equation}
\VV(\dynamics,\obj ) \subset \ds \ov{\bigcap}_{k \in \NN} \VV_k =
\lim_{k \to +\infty} \! \downarrow  \! \VV_k \; .
\label{eq:algo1}
\end{equation}
Conditions ensuring that equality holds may be found in
\citep*{Stpierre:1994}. 
Notice that, when the decreasing sequence $(\VV_{k})_{k \in \NN}$ of 
viability kernels up to time $k$ is stationary, its limit is the 
viability kernel.
Indeed, if $\VV_{k} = \VV_{k+1} $ for some $k$, then $\VV_{k}$ is a
viability domain by~\eqref{eq:induction}.
Now, by~\eqref{eq:viability_kernel}, $\VV(\dynamics,\obj ) $ is the 
largest of viability domains. 
As a consequence, $\VV_{k} = \VV(\dynamics,\obj ) $ since
$ \VV(\dynamics,\obj ) \subset \VV_{k}$ by~\eqref{eq:inclusionsVtau}.
We shall use this property in the following
Sect.~\ref{sec:Viable_control_of_generic_nonlinear_ecosystem_models}.

\section{Viable control of generic nonlinear ecosystem models with harvesting}
\label{sec:Viable_control_of_generic_nonlinear_ecosystem_models}

For a generic ecosystem model~\eqref{eq:ecosystem},
we provide an explicit description of the viability kernel.
Then, we shall specify the results for predator--prey systems, in
particular for discrete-time Lotka--Volterra models.

The acceptable set $\obj$ in~\eqref{eq:constraint} is defined by 
\emph{minimal biomass levels} 
$\stockONE\llower\geq 0$, $\stockTWO\llower\geq 0$
and \emph{minimal catch levels} 
$\catchONE\llower\geq 0$, $\catchTWO\llower\geq 0$:
\begin{equation}
\obj=\{\,(\one,\two,\controlONE,\controlTWO)\in \RR^4 \mid 
\one\geq \stockONE\llower, \; 
\two\geq \stockTWO\llower,  \; 
\controlONE \one\geq\catchONE\llower,  \; 
\controlTWO \two\geq\catchTWO\llower\,\} \; .
\label{eq:thresholds}
\end{equation}

\subsection{Expression of the viability kernel}
\label{ss:Expression_of_the_viability_kernel}

The following Proposition~\ref{pr:predator_preyBIS} gives an explicit
description of the viability kernel, under some conditions on the
minimal levels. 

\begin{proposition}
Assume that the function $\dynamicsONE:\RR^3 \to \RR$ is 
continuously decreasing in the control $\controlONE$  
and satisfies $\lim_{\controlONE \to +\infty}
\dynamicsONE(\one,\two,\controlONE)\leq 0$,
and that $\dynamicsTWO:\RR^3 \to \RR$ is 
continuously decreasing in the control variable $\controlTWO$,
and satisfies $\lim_{\controlTWO \to +\infty}
\dynamicsTWO(\one,\two,\controlTWO)\leq 0$. 
If the minimal levels in~\eqref{eq:thresholds} are such that
the following growth factors are greater than one
\begin{equation}
\dynamicsONE(\stockONE\llower,\stockTWO\llower,\frac{\catchONE\llower}{\stockONE\llower})
\geq 1
\mtext{ and } 
\dynamicsTWO(\stockONE\llower,\stockTWO\llower,\frac{\catchTWO\llower}{\stockTWO\llower})
\geq 1 \; ,
\label{eq:favorable_conditions2}
\end{equation}
the viability kernel associated with the dynamics
$\dynamics$ in~\eqref{eq:ecosystem} 
and the acceptable set $\obj$ in~\eqref{eq:thresholds} is given by
\begin{equation}
\VV(\dynamics,\obj) = \left\{(\one,\two) \mid
\one\geq \stockONE\llower, \;  \two\geq \stockTWO\llower, \;  
\one\dynamicsONE(\one,\two,\frac{\catchONE\llower}{\one})\geq \stockONE\llower, \;
\two\dynamicsTWO(\one,\two,\frac{\catchTWO\llower}{\two})\geq \stockTWO\llower 
\right\}.
\label{eq:predator_prey_viability_kernel}
\end{equation}
\label{pr:predator_preyBIS}
\end{proposition}


\begin{proof}
According to induction~\eqref{eq:induction}, we have:
\begin{eqnarray*}
\VV_0 &=& \{\,(\one,\two) \left|  \one\geq \stockONE\llower,  \two\geq \stockTWO\llower
\right. \,\},\\[3mm]
\VV_1 &=& \left\{(\one,\two) \left| 
\begin{array}{ccc}
\one\geq \stockONE\llower,  \two\geq \stockTWO\llower 
\mtext{ and, for some }\, (\controlONE,\controlTWO)\geq 0,\\ 
\controlONE \one\geq \catchONE\llower, \controlTWO \two\geq \catchTWO\llower, 
\one\dynamicsONE(\one,\two,\controlONE)\geq \stockONE\llower,
\two\dynamicsTWO(\one,\two,\controlTWO)\geq \stockTWO\llower
\end{array} \right. \right\} \\
&=& \left\{(\one,\two) \left| 
\one\geq \stockONE\llower,  \two\geq \stockTWO\llower,  
\one\dynamicsONE(\one,\two,\frac{\catchONE\llower}{\one})\geq \stockONE\llower,
\two\dynamicsTWO(\one,\two,\frac{\catchTWO\llower}{\two})\geq \stockTWO\llower \right. \right\}\\
&& \mtext{ because } \controlONE \mapsto \dynamicsONE(\one,\two,\controlONE) \mtext{ and } 
 \controlTWO \mapsto \dynamicsTWO(\one,\two,\controlTWO) \mtext{ are decreasing,}\\
&& \mtext{ and thus we may select } 
\controlONE=\frac{\catchONE\llower}{\one}, \; 
 \controlTWO =\frac{\catchTWO\llower}{\two}. \\
&& \mtext{Denoting } \one'=
\one\dynamicsONE(\one,\two,\controlONE),
\,\, \two'=\two\dynamicsTWO(\one,\two,\controlTWO), \mtext{we obtain,}\\[4mm]
\VV_2 &=& \left\{(\one,\two) \left| 
\begin{array}{l}
\one\geq \stockONE\llower,  \two\geq \stockTWO\llower\mtext{ and, for some }\, (\controlONE,\controlTWO)\geq 0,\\ 
\controlONE \one\geq \catchONE\llower,\,\, \controlTWO \two\geq \catchTWO\llower\\
\one'\geq \stockONE\llower,\,\, 
\one'\dynamicsONE(\one',\two',\frac{\catchONE\llower}{\one'})\geq \stockONE\llower,\,\, 
\two'\geq \stockTWO\llower,\,\, 
\two'\dynamicsTWO(\one',\two',\frac{\catchTWO\llower}{\two'})\geq \stockTWO\llower
\end{array} \right. \right\} \; .
\end{eqnarray*}
We shall now make use of the property, recalled in
Sect.~\ref{sec:Viability_and_sustainable_management}, that
when the decreasing sequence $(\VV_{k})_{k \in \NN}$ of 
viability kernels up to time $k$ is stationary, its limit is the 
viability kernel $ \VV(\dynamics,\obj )$.
Hence, it suffices to show that $\VV_1\subset \VV_2$ to obtain that
$ \VV(\dynamics,\obj) = \VV_1$.
Let $(\one,\two)\in \VV_1$, so that
\begin{equation*}
\one\geq \stockONE\llower, \quad \two\geq \stockTWO\llower \mtext{ and }~ 
\one\dynamicsONE(\one,\two,\frac{\catchONE\llower}{\one})\geq \stockONE\llower, \quad 
\two\dynamicsTWO(\one,\two,\frac{\catchTWO\llower}{\two})\geq \stockTWO\llower \; .
\end{equation*}
Since $\dynamicsONE:\RR^3 \to \RR$ is
continuously decreasing in the control variable, with
$\lim_{\controlONE \to +\infty}\dynamicsONE(\one,\two,\controlONE)\leq 0$, 
and since $\one\dynamicsONE(\one,\two,\frac{\catchONE\llower}{\one})
\geq \stockONE\llower$, there exists a $\hat{\controlONE} \geq
\frac{\catchONE\llower}{\one}$ (depending on $\one$ and $\two$) such that
$\one'=\one\dynamicsONE(\one,\two,\hat{\controlONE})= \stockONE\llower$.
The same holds for $\dynamicsTWO:\RR^3 \to \RR$ and
$\two'=\two\dynamicsTWO(\one,\two,\hat{\controlTWO})= \stockTWO\llower$.
By~\eqref{eq:favorable_conditions2}, we deduce that
\[
\one'\dynamicsONE(\one',\two',\frac{\catchONE\llower}{\one'}) = \stockONE\llower
\dynamicsONE(\stockONE\llower,\stockTWO\llower,\frac{\catchONE\llower}{\stockONE\llower})
\geq \stockONE\llower \mtext{ and }
\two'\dynamicsTWO(\one',\two',\frac{\catchTWO\llower}{\two'}) = \stockTWO\llower
\dynamicsTWO(\stockONE\llower,\stockTWO\llower,\frac{\catchTWO\llower}{\stockTWO\llower})
\geq \stockTWO\llower \; .
\]
The inclusion $\VV_1\subset \VV_2$ follows.
\end{proof}

\begin{corollary}
Suppose that the assumptions of Proposition~\ref{pr:predator_prey} are
satisfied. 
Denoting 
\[
\left\{ \begin{array}{rcl}
\hat{\controlONE}(\one,\two) &=& \max \{ \controlONE \geq \frac{\catchONE\llower}{\one} \mid
\one\dynamicsONE(\one,\two,\controlONE)= \one\llower \} \; , \\[3mm]
\hat{\controlTWO}(\one,\two) &=& \max \{ \controlTWO \geq \frac{\catchTWO\llower}{\two} \mid
\two\dynamicsTWO(\one,\two,\controlTWO)=\two\llower \} \; ,
\end{array} \right.
\]
the set of viable controls is given by
\begin{equation*}
\UU_{\VV(\dynamics,\obj)}(\one,\two)= \left\{(\controlONE,\controlTWO) \left| 
\begin{array}{l}
\hat{\controlONE}(\one,\two) \geq \controlONE\geq \frac{\catchONE\llower}{\one}, \quad 
\hat{\controlTWO}(\one,\two) \geq \controlTWO\geq \frac{\catchTWO\llower}{\two}, \\
\one'\dynamicsONE(\one',\two',\frac{\catchONE\llower}{\one'})\geq \one\llower,\,\, 
\two'\dynamicsTWO(\one',\two',\frac{\catchTWO\llower}{\two'})\geq \two\llower
\end{array} \right. \right\}  \; ,
\end{equation*}
where 
$\one'= \one\dynamicsONE(\one,\two,\controlONE),\,\, \two'=\two\dynamicsTWO(\one,\two,\controlTWO) $. 
\label{cor:viable-contr}
\end{corollary}

\subsection{Proof of
  Proposition~\ref{pr:favorable_conditions_equilibrium}}

\begin{proof}
By~\eqref{eq:conditions_initial_point}, and the property that 
both $\dynamicsONE$ and $\dynamicsTWO$ are decreasing in the control
variable, the quantities~\eqref{eq:state_maximal_catches} exist.

Also since both $\dynamicsONE$ and $\dynamicsTWO$ are decreasing in the control variable, we obtain that 
\begin{equation*}
\dynamicsONE(\stockONE\llower,\stockTWO\llower,
\frac{\catchONE\lloweropt(\one_0,\two_0)}{\stockONE\llower})
\geq
\dynamicsONE(\stockONE\llower,\stockTWO\llower,
\frac{\catchONE\lloweropt}{\stockONE\llower}) = 1 
\mtext{ and } 
\dynamicsTWO(\stockONE\llower,\stockTWO\llower,
\frac{\catchTWO\lloweropt(\one_0,\two_0)}{\stockTWO\llower})\geq 
\dynamicsTWO(\stockONE\llower,\stockTWO\llower,
\frac{\catchTWO\lloweropt}{\stockTWO\llower}) = 1 \; .
\end{equation*}
To end up, the above inequalities and 
the assumption that 
$ \one_0 \geq \stockONE\llower$ and $\two_0 \geq \stockTWO\llower$
allow us to conclude, thanks to Proposition~\ref{pr:predator_preyBIS}, that $(\one_0,\two_0)$ belongs to the viability kernel 
\( \VV(\dynamics,\obj) \)
given in~\eqref{eq:predator_prey_viability_kernel}.

In other words, starting from the initial point 
\( (\one(t_0),\two(t_0)) = (\one_0,\two_0) \),
there exists an appropriate harvesting path 
which can provide, for all time, at least the
catches~\eqref{eq:state_maximal_catches}. 
\end{proof}

\newcommand{\noopsort}[1]{} \ifx\undefined\allcaps\def\allcaps#1{#1}\fi


\begin{thebibliography}{35}
\providecommand{\natexlab}[1]{#1}
\providecommand{\url}[1]{\texttt{#1}}
\expandafter\ifx\csname urlstyle\endcsname\relax
  \providecommand{\doi}[1]{doi: #1}\else
  \providecommand{\doi}{doi: \begingroup \urlstyle{rm}\Url}\fi

\bibitem[Alheit and {N}iquen(2004)]{Alheit-Niquen:2004}
J{\"{u}}rgen Alheit and Miguel~\ {N}iquen.
\newblock Regime shifts in the {Humboldt Current Ecosystem}.
\newblock \emph{Progress in Oceanography}, 60:\penalty0 201--222, 2004.

\bibitem[Aubin(1991)]{Aubin:1991}
J-P. Aubin.
\newblock \emph{Viability Theory}.
\newblock Birkh{\"a}user, Boston, 1991.
\newblock 542 pp.

\bibitem[Ball{\'o}n et~al.(2008)Ball{\'o}n, Wosnitza-Mendo, Guevara-Carrasco,
  and Bertrand]{Ballon-Wosnitza-Mendo-Guevara-Carrasco-Bertrand:2008}
Michael Ball{\'o}n, Claudia Wosnitza-Mendo, Renato Guevara-Carrasco, and Arnaud
  Bertrand.
\newblock The impact of overfishing and {El Ni{\~n}o} on the condition factor
  and reproductive success of {Peruvian} hake, {Merluccius} gayi peruanus.
\newblock \emph{Progress In Oceanography}, 79\penalty0 (2-4):\penalty0 300 --
  307, 2008.

\bibitem[B\'en\'e and Doyen(2003)]{Doyen-Bene:2003}
C.~B\'en\'e and L.~Doyen.
\newblock Sustainability of fisheries through marine reserves: a robust
  modeling analysis.
\newblock \emph{Journal of Environmental Management}, 69\penalty0 (1):\penalty0
  1--13, 2003.

\bibitem[B\'en\'e et~al.(2001)B\'en\'e, Doyen, and
  Gabay]{Bene-Doyen-Gabay:2001}
C.~B\'en\'e, L.~Doyen, and D.~Gabay.
\newblock A viability analysis for a bio-economic model.
\newblock \emph{Ecological Economics}, 36:\penalty0 385--396, 2001.

\bibitem[Bertrand et~al.(2004)Bertrand, Segura, Guti\'{e}rrez, and
  V\'{a}squez]{Bertrand_etal:2004}
Arnaud Bertrand, Marceliano Segura, Mariano Guti\'{e}rrez, and Luis
  V\'{a}squez.
\newblock From small-scale habitat loopholes to decadal cycles: a habitat-based
  hypothesis explaining fluctuation in pelagic fish populations off {Peru}.
\newblock \emph{Fish and Fisheries}, 5:\penalty0 296--316, 2004.

\bibitem[Bertsekas and Rhodes(1971)]{bertsekas71minimax}
D.~Bertsekas and I.~Rhodes.
\newblock On the minimax reachability of target sets and target tubes.
\newblock \emph{Automatica}, 7:\penalty0 233--247, 1971.

\bibitem[Chapel et~al.(2008)Chapel, Deffuant, Martin, and
  Mullon]{Chapel-Deffuant-Martin-Mullon:2008}
Laetitia Chapel, Guillaume Deffuant, Sophie Martin, and Christian Mullon.
\newblock Defining yield policies in a viability approach.
\newblock \emph{Ecological Modelling}, 212\penalty0 (1-2):\penalty0 10 -- 15,
  2008.

\bibitem[Checkley et~al.(2009)Checkley, Alheit, Ooseki, and Roy]{Checkley:2009}
David~M. Checkley, J{\"{u}}rgen Alheit, Yoshioki Ooseki, and Claude Roy.
\newblock \emph{Climate Change and Small Pelagic Fish}.
\newblock Cambridge University Press, 2009.

\bibitem[Clark(1990)]{Clark:1990}
C.~W. Clark.
\newblock \emph{Mathematical Bioeconomics}.
\newblock Wiley, New York, second edition, 1990.

\bibitem[Clarke et~al.(1995)Clarke, Ledayev, Stern, and
  Wolenski]{Clarkeetal:1995}
F.~H. Clarke, Y.~S. Ledayev, R.~J. Stern, and P.~R. Wolenski.
\newblock Qualitative properties of trajectories of control systems: a survey.
\newblock \emph{Journal of Dynamical Control Systems}, 1:\penalty0 1--48, 1995.

\bibitem[{De Lara} and Doyen(2008)]{DeLara-Doyen:2008}
M.~{De Lara} and L.~Doyen.
\newblock \emph{Sustainable Management of Natural Resources. Mathematical
  Models and Methods}.
\newblock Springer-Verlag, Berlin, 2008.

\bibitem[{De Lara} and Martinet(2009)]{DeLara-Martinet:2009}
M.~{De Lara} and V.~Martinet.
\newblock Multi-criteria dynamic decision under uncertainty: A stochastic
  viability analysis and an application to sustainable fishery management.
\newblock \emph{Mathematical Biosciences}, 217\penalty0 (2):\penalty0 118--124,
  February 2009.

\bibitem[{D}e {Lara} et~al.(2007){D}e {Lara}, Doyen, Guilbaud, and
  {R}ochet]{DeLara-Doyen-Guilbaud-Rochet_IJMS:2007}
M.~{D}e {Lara}, L.~Doyen, T.~Guilbaud, and M.-J. {R}ochet.
\newblock {Is a management framework based on spawning-stock biomass indicators
  sustainable? A viability approach}.
\newblock \emph{ICES J. Mar. Sci.}, 64\penalty0 (4):\penalty0 761--767, 2007.

\bibitem[Doyen and {De Lara}(2010)]{Doyen-DeLara:2010}
L.~Doyen and M.~{De Lara}.
\newblock Stochastic viability and dynamic programming.
\newblock \emph{Systems and Control Letters}, 59\penalty0 (10):\penalty0
  629--634, October 2010.

\bibitem[Eisenack et~al.(2006)Eisenack, Sheffran, and
  Kropp]{Eisenack-Sheffran-Kropp:2006}
K.~Eisenack, J.~Sheffran, and J.~Kropp.
\newblock The viability analysis of management frameworks for fisheries.
\newblock \emph{Environmental Modeling and Assessment}, 11\penalty0
  (1):\penalty0 69--79, February 2006.

\bibitem[FAO(1999)]{FAO:1999}
FAO.
\newblock Indicators for sustainable development of marine capture fisheries.
\newblock FAO Technical Guidelines for Responsible Fisheries~8, FAO, 1999.
\newblock 68~pp.

\bibitem[Garcia et~al.(2003)Garcia, Zerbi, Aliaume, Chi, and
  Lasserre]{Garcia-Zerbi-Aliaume-DoChi-Lasserre:2003}
S.~Garcia, A.~Zerbi, C.~Aliaume, T.~Do Chi, and G.~Lasserre.
\newblock The ecosystem approach to fisheries. {Issues}, terminology,
  principles, institutional foundations, implementation and outlook.
\newblock \emph{FAO Fisheries Technical Paper}, 443\penalty0 (71), 2003.

\bibitem[Hall(1988)]{Hall:1988}
C.~Hall.
\newblock An assessment of several of the historically most influential
  theoretical models used in ecology and of the data provided in their support.
\newblock \emph{Ecological Modelling}, 43\penalty0 (1-2):\penalty0 5--31, 1988.

\bibitem[Hollowed et~al.(2000)Hollowed, Bax, Beamish, Collie, Fogarty,
  Livingston, Pope, and
  Rice]{Hollowed-Bax-Beamish-Collie-Fogarty-Livingston-Pope-Rice:2000}
Anne~B. Hollowed, Nicholas Bax, Richard Beamish, Jeremy Collie, Michael
  Fogarty, Patricia Livingston, John Pope, and Jake~C. Rice.
\newblock {Are multispecies models an improvement on single-species models for
  measuring fishing impacts on marine ecosystems?}
\newblock \emph{ICES J. Mar. Sci.}, 57\penalty0 (3):\penalty0 707--719, 2000.

\bibitem[ICES(2004)]{ICES:2004}
ICES.
\newblock Report of the {ICES} advisory committee on fishery management and
  advisory committee on ecosystems, 2004.
\newblock {ICES} {Advice}, 1, ICES, 2004.
\newblock 1544 pp.

\bibitem[IMARPE(2000)]{IMARPE:2000}
IMARPE.
\newblock Trabajos expuestos en el taller internacional sobre la anchoveta
  peruana ({TIAP}), 9-12 {Mayo} 2000.
\newblock \emph{Bol. Inst. Mar Peru}, 19:\penalty0 1--2, 2000.

\bibitem[IMARPE(2004)]{IMARPE:2004}
IMARPE.
\newblock Report of the first session of the international panel of experts for
  assessment of {Peruvian} hake population. {March} 2003.
\newblock \emph{Bol. Inst. Mar Peru}, 21:\penalty0 33--78, 2004.

\bibitem[Katz et~al.(2003)Katz, Zabel, Harvey, Good, and
  Levin]{Katz-Zabel-Harvey-Good-Levin:2003}
Stephen~L. Katz, Richard Zabel, Chris Harvey, Thomas Good, and Phillip Levin.
\newblock Ecologically sustainable yield.
\newblock \emph{American Scientist}, 91\penalty0 (2):\penalty0 150, March-April
  2003.

\bibitem[Larkin(1977)]{Larkin:1977}
P.~A. Larkin.
\newblock {An Epitaph for the Concept of Maximum Sustained Yield}.
\newblock \emph{Transactions of the American Fisheries Society}, 106\penalty0
  (1):\penalty0 1--11, January 1977.

\bibitem[Mullon et~al.(2004)Mullon, Cury, and Shannon]{Mullonetal:2004}
C.~Mullon, P.~Cury, and L.~Shannon.
\newblock Viability model of trophic interactions in marine ecosystems.
\newblock \emph{Natural Resource Modeling}, 17:\penalty0 27--58, 2004.

\bibitem[Murray(2002)]{Murray:2002}
J.~D. Murray.
\newblock \emph{Mathematical Biology}.
\newblock Springer-Verlag, Berlin, third edition, 2002.

\bibitem[Pauly and Palomares(1989)]{Pauly-Palomares:1989}
Daniel Pauly and Mar{\'\i}a~Lourdes Palomares.
\newblock \emph{New estimates of monthly biomass, recruitment and related
  statistics of anchoveta (\emph{Engraulis ringens}) off {Peru}
  (4-14$\,^{\circ}\mathrm{S}$), 1953-1985}, pages 189--206.
\newblock 18. ICLARM Conference Proceedings, 1989.

\bibitem[PRODUCE(2005)]{PRODUCE:2005}
PRODUCE.
\newblock Establecen r\'egimen provisional de pesca del recurso merluza
  correspondiente al a\~no 2006.
\newblock \emph{El Peruano}, page 307804, 30 de diciembre 2005.
\newblock RM-356-2005-PRODUCE.

\bibitem[PRODUCE(2006)]{PRODUCE:2006a}
PRODUCE.
\newblock Establecen r\'egimen provisional de pesca del recurso merluza
  correspondiente al a\~no 2007.
\newblock \emph{El Peruano}, page 335485, 27 de diciembre 2006.
\newblock RM-357-2006-PRODUCE.

\bibitem[Rapaport et~al.(2006)Rapaport, Terreaux, and
  Doyen]{Rapaport-Terreaux-Doyen:2006}
A.~Rapaport, J.-P. Terreaux, and L.~Doyen.
\newblock Sustainable management of renewable resource: a viability approach.
\newblock \emph{Mathematics and Computer Modeling}, 43\penalty0 (5-6):\penalty0
  466--484, March 2006.

\bibitem[Saint-Pierre(1994)]{Stpierre:1994}
P.~Saint-Pierre.
\newblock Approximation of viability kernel.
\newblock \emph{Applied Mathematics and Optimization}, 29:\penalty0 187--209,
  1994.

\bibitem[Schaefer(1954)]{Schaefer:1954}
M.~B. Schaefer.
\newblock Some aspects of the dynamics of populations important to the
  management of commercial marine fisheries.
\newblock \emph{Bulletin of the Inter-American tropical tuna commission},
  1:\penalty0 25--56, 1954.

\bibitem[Sun et~al.(2001)Sun, Chiang, Liu, and Chang]{Sun_etal:2001}
Chin-Hwa Sun, Fu-Sung Chiang, Te-Shi Liu, and Ching-Cheng Chang.
\newblock A welfare analysis of {El Ni{\~n}o} forecasts in the international
  trade of fish meal - an application of stochastic spatial equilibrium model.
\newblock 2001 Annual meeting, August 5-8, Chicago, IL 20770, American
  Agricultural Economics Association (New Name 2008: Agricultural and Applied
  Economics Association), 2001.

\bibitem[Tam et~al.(2006)Tam, Purca, Duarte, Blaskovic, and Espinoza]{Tam:2006}
J.~Tam, S.~Purca, L.~O. Duarte, V.~Blaskovic, and P.~Espinoza.
\newblock Changes in the diet of hake associated with {El Ni{\~n}o} 1997-1998
  in the {Northern Humboldt Current} ecosystem.
\newblock \emph{Advances in Geosciences.}, 6:\penalty0 63--67, 2006.

\end{thebibliography}
\end{document}